\documentclass{amsart}
\usepackage{graphicx}

\newtheorem{theorem}{Theorem}[section]
\newtheorem{lemma}[theorem]{Lemma}

\theoremstyle{definition}
\newtheorem{definition}[theorem]{Definition}

\newtheorem{corollary}[theorem]{Corollary}
\newtheorem{comment}[theorem]{Comment}
\newtheorem{conjecture}[theorem]{Conjecture}

\newtheorem{proposition}[theorem]{Proposition}

\theoremstyle{remark}

\numberwithin{equation}{section}

\begin{document}

\title{Efficient prime counting\\ and the Chebyshev primes}

\author{Michel Planat}
\address{Institut FEMTO-ST, CNRS, 32 Avenue de l'Observatoire, F-25044 Besan\c con, France.}
\email{michel.planat@femto-st.fr}

\author{Patrick Sol\'e}
\address{Telecom ParisTech, 46 rue Barrault, 75634 Paris Cedex 13, France.}
\curraddr{MECAA, King Abdulaziz University, Jeddah, Saudi Arabia.}
\email{sole@telecom-paristech.fr}


\date{October 14\normalfont, 2011 and, in revised form}



\begin{abstract}

The function $\epsilon(x)=\mbox{li}(x)-\pi(x)$ is known to be positive up to the (very large) Skewes' number. Besides, according to Robin's work, the functions $\epsilon_{\theta}(x)=\mbox{li}[\theta(x)]-\pi(x)$ and $\epsilon_{\psi}(x)=\mbox{li}[\psi(x)]-\pi(x)$ are positive if and only if Riemann hypothesis (RH) holds (the first and the second Chebyshev function are $\theta(x)=\sum_{p \le x} \log p$ and $\psi(x)=\sum_{n=1}^x \Lambda(n)$, respectively, $\mbox{li}(x)$ is the logarithmic integral, $\mu(n)$ and $\Lambda(n)$ are the M\"obius and the Von Mangoldt functions). Negative jumps in the above functions $\epsilon$, $\epsilon_{\theta}$ and $\epsilon_{\psi}$ may potentially occur only at $x+1 \in \mathcal{P}$ (the set of primes). One denotes $j_p=\mbox{li}(p)-\mbox{li}(p-1)$ and one investigates the jumps $j_p$, $j_{\theta(p)}$ and $j_{\psi(p)}$. In particular, $j_p<1$, and $j_{\theta(p)}>1$ for $p<10^{11}$. Besides, $j_{\psi(p)}<1$ for any odd $p \in \mathcal{\mbox{Ch}}$, an infinite set of so-called {\it Chebyshev primes
 } with partial list
 $\{109, 113, 139, 181, 197, 199, 241, 271, 281, 283, 293, 313, 317, 443, 449, 461, 463, \ldots\}$.

We establish a few properties of the set $\mathcal{\mbox{Ch}}$, give accurate approximations of the jump $j_{\psi(p)}$ and relate the derivation of $\mbox{Ch}$ to the explicit Mangoldt formula for $\psi(x)$. In the context of RH, we introduce the so-called {\it Riemann primes} as champions of the function $\psi(p_n^l)-p_n^l$ (or of the function $\theta(p_n^l)-p_n^l$ ). Finally, we find a {\it good} prime counting function $S_N(x)=\sum_{n=1}^N \frac{\mu(n)}{n}\mbox{li}[\psi(x)^{1/n}]$, that is found to be much better than the standard Riemann prime counting function.

\end{abstract}

\maketitle

\section*{Introduction}

Let us introduce the first and the second Chebyshev function $\theta(x)=\sum_{p \le x} \log p$ (where $p \in \mathcal{P}$: the set of prime numbers) and $\psi(x)=\sum_{n=1}^x \Lambda(n)$, the logarithmic integral $\mbox{li}(x)$, the M\"obius function $\mu(n)$ and the Von Mangoldt function $\Lambda(n)$ \cite{Edwards74,Hardy79}. The number of primes up to $x$ is denoted $\pi(x)$. Indeed, $\theta(x)$ and $\psi(x)$ are the logarithm of the product of all primes up to $x$, and the logarithm of the least common multiple of all positive integers up to $x$, respectively.

It has been known for a long time that $\theta(x)$ and $\psi(x)$ are asymptotic to $x$ (see \cite{Hardy79}, p. 341). There also exists an explicit formula, due to Von Mangoldt, relating $\psi(x)$ to the non-trivial zeros $\rho$ of the Riemann zeta function $\zeta(s)$ \cite{Edwards74,Davenport80}. One defines the normalized Chebyshev function $\psi_0(x)$ to be $\psi(x)$ when $x$ is not a prime power, and $\psi(x)-\frac{1}{2}\Lambda(x)$ when it is. The explicit Von Mangoldt formula reads
%
$$\psi_0(x)=x-\sum_{\rho} \frac{x^{\rho}}{\rho}-\frac{\zeta'(0)}{\zeta(0)}-\frac{1}{2}\log (1-x^{-2}), ~\mbox{for}~x>1.$$
%

The function $\epsilon(x)=\mbox{li}(x)-\pi(x)$ is known to be positive up to the (very large) Skewes' number \cite{Skewes}.
In this paper we are first interested in the jumps (they occur at primes $p$) in the function $\epsilon_{\theta(x)}=\mbox{li}[\theta(x)]-\pi(x)$. Following Robin's work on the relation between $\epsilon_{\theta(x)}$ and RH (Theorem \ref{Thm1}), this allows us to derive a new statement (Theorem \ref{Thm2}) about the jumps of $\mbox{li}[\theta(p)]$ and Littlewood's oscillation theorem.

Then, we study the refined function $\epsilon_{\psi(x)}=\mbox{li}[\psi(x)]-\pi(x)$ and we observe that the sign of the jumps of $\mbox{li}[\psi(p)]$ is controlled by an infinite sequence of primes that we call the {\it Chebyshev primes} $\mbox{Ch}_n$ (see proposition \ref{chebn}). The primes $\mbox{Ch}_n$ (and the generalized primes $\mbox{Ch}_n^{(l)}$) are also obtained by using an accurate calculation of the jumps of $\mbox{li}[\psi(p)]$, as in conjecture \ref{newdef} (and of the jumps of the function $\mbox{li}[\psi(p^l)]$, as in conjecture \ref{newdef2}). One conjectures that the function $\mbox{Ch}_n-p_{2n}$ has infinitely many zeros. There exists a potential link between the non-trivial zeros $\rho$ of $\zeta(s)$ and the position of the $\mbox{Ch}_n^{(l)}$'s that is made quite explicit in Sec. \ref{sec14} (conjecture \ref{defMan}), and in Sec. \ref{Riemannp} in our definition of the Riemann primes. In  this context, we contribute to the Sloane's encyclopedia with integer sequences \footnote{The relevant sequences are A196667 to 196675 (related to the Chebyshev primes), A197185 to A197188 (related to the Riemann primes of the $\psi$-type and A197297 to A197300 (related to the Riemann primes of the $\theta$-type.}.

Finally, we introduce a new prime counting function $R(x)=\sum_{n>1}\frac{\mu(n)}{n}\mbox{li}(x^{1/n})$, better than the standard Riemann's one, even with three terms in the expansion.

\section{Selected results about the functions $\epsilon$, $\epsilon_{\theta}$, $\epsilon_{\psi}$}

Let $p_n$ be the $n$-th prime number and $j(p_n)=\mbox{li}(p_n)-\mbox{li}(p_n-1)$ be the jump in the logarithmic integral at $p_n$. For any $n>2$ one numerically observes that $j_{p_n}<1$.
This statement is not useful for the rest of the paper. But it is enough to observe that $j_5=0.667\ldots$ and  that the sequence $j_{p_n}$ is strictly decreasing.

The next three subsections deal with the jumps in the function $\mbox{li}[\theta(x)]$ and $\mbox{li}[\psi(x)]$.

\subsection{The jumps in the function $\mbox{li}[\theta(x)]$}


\begin{theorem}\label{Thm1}
(Robin). The statement $\epsilon_{\theta(x)}=\mbox{li}[\theta(x)]-\pi(x)>0$ is equivalent to RH \cite{Robin84,Sandor95}.
\end{theorem}

\begin{corollary}
 (related to Robin \cite{Robin84}). The statement $\epsilon_{\psi(x)}=\mbox{li}[\psi(x)]-\pi(x)>0$ is equivalent to RH.
\end{corollary}
\begin{proof}
If RH is true then, using the fact $\psi(x)>\theta(x)$ and that $\mbox{li}(x)$ is a strictly growing function when $x >  1$, this follows from theorem 1 in Robin \cite{Robin84}. If RH is false, Lemma 2 in Robin ensures the violation of the inequality.  
\end{proof}


Using the fact that $\theta(p_{n+1}-1)=\theta(p_n)$, define the jump of index $n$ as
$$J_n=j_{\theta(p_n)}=\mbox{li}[\theta(p_{n+1})]-\mbox{li}[\theta(p_{n})]=\int_{\theta(p_n)}^{\theta(p_{n+1})}\frac{dt}{\log t}.$$

\begin{proposition}
If $p_{n+1}< {10}^{11},$ then $J_n=j_{\theta(p_n)}>1.$
\end{proposition}
\begin{proof}
The integral definition of the jump yields
$$J_n\ge \frac{\theta(p_{n+1})-\theta(p_n)}{\log \theta(p_{n+1})}=\frac{\log p_{n+1}}{\log \theta(p_{n+1})}.$$
The result now follows after observing that by \cite[Theorem18]{RS}, we have $\theta (x)<x$ for $x < {10}^8$, and by using the note added in proof of \cite{Schoen76} that establishes that $\theta(x)<x$ for $x<10^{11}$.
\end{proof}
By seeing this result it would be natural to make the
\begin{conjecture}\label{conjjump}
$\forall n\ge 1$ we have $J_n>1.$
\end{conjecture}

However, building on Littlewood's oscillation theorem for $\theta$ we can prove that $J_n$ oscillates about $1$ with a small amplitude.
Let us recall the Littlewood's oscillation theorem \cite[Theorem 6.3, p.200]{EM},\cite[Theorem 34]{I}
$$\theta(x)-x=\Omega_{\pm}(x^{1/2}\log_3 x),~\mbox{when}~x \rightarrow \infty,$$ 
where $\log_3 x=\log \log \log x.$ The omega notations means that there are infinitely many numbers $x$, and constants $C_+$ and $C_-$, satisfying

$$\theta(x) \le x-C_-\sqrt{x}\log_3 x~\mbox{or}~\theta(x) \ge x+ C_+\sqrt{x}\log_3 x.$$
We now prepare the proof of the invalidity of conjecture (\ref{conjjump}) by writing two lemmas.
\begin{lemma} \label{encad}
For $n\ge 1,$ we have the bounds $$\frac{\log p_{n+1}}{\log \theta(p_{n+1})} \le J_n\le \frac{\log p_{n+1}}{\log \theta(p_{n})}.$$
\end{lemma}
\begin{proof}
This is straightforward from the integral definition of the jump.
\end{proof}

\begin{lemma}\label{oscillo}
For $n$ large, we have $$ \theta(p_{n+1})=p_{n+1}+\Omega_{\pm}\left(\sqrt{p_{n+1}}\log_3 p_{n+1}\right).$$
\end{lemma}
\begin{proof}
We know that by \cite[Theorem 6.3, p.200]{EM}, we have for $x>0$ and large
$$ \theta(x)-x=\Omega_{\pm}\left(\sqrt{x}\log_3 x\right).$$
The result follows by considering the primes closest to $x.$
\end{proof}

We can now state and prove the main result of this section.
\begin{theorem}\label{Thm2}
For $n$ large we have $$J_n=1+\Omega_{\pm}\left( \frac{\log_3 p_{n+1}}{\sqrt{p_{n+1}}\log p_{n+1}}\right).$$
 \end{theorem}
\begin{proof}
By lemma \ref{oscillo} we know there is a constant $C_-$ such that for infinitely many $n$'s we have
$$ \theta(p_{n+1})< p_{n+1}-{C_-}\sqrt{p_{n+1}}\log_3 p_{n+1}.$$
By combining with the first inequality of lemma \ref{encad} and writing
$$ \log p_{n+1}=\log p_{n}+\log \left( 1- {C_-}\frac{\log_3 p_{n+1}}{\sqrt{p_{n+1}}}\right),$$
the minus part of the statement follows after some standard asymptotics. To prove the plus part write $\theta(p_n)=\theta(p_{n+1})-\log p_{n+1},$ and proceed as before.
\end{proof}


\subsection{The jumps in the function $\mbox{li}[\psi(x)]$ and the Chebyshev primes}

\begin{definition}\label{def1}
 Let $p\in \mathcal{P}$ be a odd prime number and  the function $j_{\psi(p)}=\mbox{li}[\psi(p)]-\mbox{li}[\psi(p-1)]$. The primes $p$ such that  $j_{\psi(p)}<1$ are called here {\it Chebyshev primes} $\mbox{Ch}_n$ 
\footnote{Our terminology should not be confused with that used in \cite{Cusick05} where the {\it Tchebychev} primes are primes of the form $4n2^m+1$, with $m>0$ and $n$ an odd prime. We used the Russian spelling Chebyshev to distinguish both meanings.}. In increasing order, they are \cite[Sequence A196667]{Sloane}

 $\{109, 113, 139, 181, 197, 199, 241, 271, 281, 283, 293, 313, 317, 443, 449, 461, 463, \ldots\}$.
\end{definition} 
 
\begin{comment}
The number of Chebyshev primes less than $10^n$, $(n=1,2,\ldots)$ is the sequence $\{0,0,42,516,4498,41423\ldots\}$ \cite[Sequence A196671]{Sloane}. This sequence suggests the density $\frac{1}{2}\pi(x)$ for the
 Chebyshev primes. The largest gaps between the Chebyshev primes are
 
$\{4,26,42,126,146,162,176,470,542,1370,1516,4412,8196,14928,27542,30974$,

$51588,62906\ldots\},$ \cite[Sequence A196672]{Sloane}
and the Chebyshev primes that begin a record gap to the next Chebyshev prime are

\noindent
$\{109,113,139,317,887,1327,1913,3089,8297,11177,29761,45707,113383,164893,$

$291377,401417,638371,1045841\ldots\}$ \cite[Sequence A196673]{Sloane}.

\end{comment}

The results for the jumps of the function $\mbox{li}[\psi(p)]$ are quite analogous to the results for the jumps of the function $\mbox{li}[\theta(p)]$ ans stated below without proof.

Let us define the $n$-th jump at a prime as
$$K_n=\mbox{li}[\psi(p_{n+1})]-\mbox{li}[\psi(p_{n+1}-1)]=\int_{\psi(p_{n+1}-1)}^{\psi(p_{n+1})} \frac{dt}{\log t}.$$

\begin{theorem}\label{thmpsi}
For $n$ large, we have $$K_n=1+\Omega_{\pm}\left(\frac{\log_3 p_{n+1}}{\sqrt{p_{n+1}\log p_{n+1}}}\right).$$
\end{theorem}

\begin{corollary}\label{chebn}
There are infinitely many Chebyshev primes $\mbox{Ch}_n$.
\end{corollary}




One observes that the sequence $\mbox{Ch}_n$ oscillates around $p_{2n}$ and the largest deviations from $p_{2n}$ seem to be unbounded at large $n$. This behaviour is illustrated in Fig 1. Based on this numerical results, we are led to the conjecture

\begin{figure}
\centering
\includegraphics{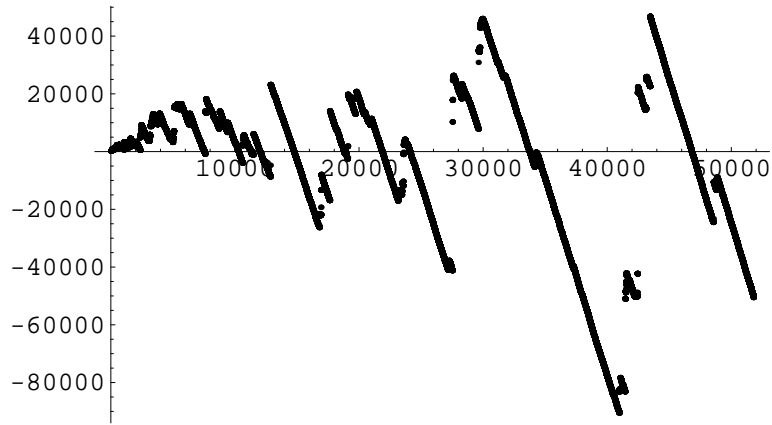}
\caption{A plot of the function $\mbox{Ch}_n-p_{2n}$ up to the $10^5$-th prime}
\end{figure}

\begin{conjecture}\label{conj111}
The function $\mbox{Ch}_n -p_{2n}$ possesses infinitely many zeros. 
\end{conjecture}

\begin{comment}
The first eleven zeros of $\mbox{Ch}_n -p_{2n}$ occur at the indices

\noindent
$\{510,10271,11259,11987,14730,18772,18884,21845,24083,33723,46789\}$ \cite[Sequence A1966674]{Sloane}
 where the corresponding Chebyshev primes are \cite[Sequence A1966675]{Sloane}
 
 \noindent
  $\{164051,231299,255919,274177,343517,447827,450451,528167,587519,847607,1209469\}.$

%
%
\end{comment}

\begin{conjecture}\label{newdef}
The jump at primes $p_n$ of the function $\mbox{li}[\psi(x)]$ may be written as $K_{n-1}=\tilde{K}_{n-1}+O(1/p_n^2).$ with $\tilde{K}_{n-1}=\frac{\log p_n}{\log\left[(\psi(p_n)+\psi(p_n-1))/2\right]}$. In particular, the sign of $\tilde{K}_{n-1}-1$ is that of $K_n-1$.
\end{conjecture}

\begin{comment}
The jump of index $n-1$ (at the prime number $p_n$) is

$$K_{n-1}=\int_{\psi(p_n-1)}^{\psi(p_n)}\frac{dt}{\log t}.$$
There exists a real $c_n$ depending of the index $n$, with $\psi(p_n-1)<c_n<\psi(p_n)$ such that
$$K_{n-1}=\frac{\psi(p_n)-\psi(p_n-1)}{\log c_n}=\frac{\log p_n}{\log c_n}.$$
Using the known locations of the Chebyshev primes of low index, it is straightforward to check that the real $c_n$ reads
$$c_n \sim \frac{1}{2}\left[\psi(p_n)+\psi(p_n-1)\right].$$

This numerical calculations support our conjecture (\ref{newdef}) that the Chebyshev primes may be derived from $\tilde{K}_{n-1}$ instead of $K_{n-1}$.

\end{comment}


\subsection{The generalized Chebyshev primes}

\begin{definition}\label{general}
 Let $p\in \mathcal{P}$ be a odd prime number and  the function $j_{\psi(p^l)}=\mbox{li}[\psi(p^l)]-\mbox{li}[\psi(p^l-1)]$, $l\ge 1$. The primes $p$ such that  $j_{\psi(p^l)}<1/l$ are called here {\it generalized Chebyshev primes} $\mbox{Ch}_n^{(l)}$ (or Chebyshev primes of index $l$).
\end{definition}

A short list of Chebyshev primes of index $2$ is as follows

$\{17,29,41,53,61,71,83,89,101,103,113,127,137,149,151,157,193,211,239,241,\ldots\}$
 \cite[Sequence A196668]{Sloane}.  

A short list of Chebyshev primes of index $3$ is as follows

$\{11,19,29,61,71,97,101,107,109,113,127,131,149,151,173,181,191,193,197,199,\ldots\}$ 
\cite[Sequence A196669]{Sloane}. 

A short list of Chebyshev primes of index $4$ is as follows

$\{5,7,17,19,31,37,41,43,53,59,67,73,79,83,101,103,107,\ldots\} $
 \cite[Sequence A196670]{Sloane}. 



\begin{conjecture}\label{newdef2}

The jump at power of primes $p_n^l$ of the function $\mbox{li}[\psi(x)]$ may be written as $K_{n-1}^{(l)}=\tilde{K}_{n-1}^{(l)}+O(1/p_n^{2l})$, with $\tilde{K}_{n-1}^{(l)}=\frac{\log p_n}{\log\left[(\psi(p_n^l)+\psi(p_n^l-1))/2\right]}$. In particular the sign of $\tilde{K}_{n-1}^{(l)}-1$ is that of $K_{n-1}^{(l)}-1$.
\end{conjecture}

\begin{comment}
Our comment is similar to the comment given in the context of proposition (\ref{newdef}) but refers to the generalized Chebyshev primes $\mbox{Ch}_n^{(l)}$.
To summarize, the jump of the function $\mbox{li}[\psi(x)]$ is accurately defined by a {\it generalized Mangoldt function} $\Lambda_n^{new}$ that is $\tilde{K}_{n-1}^{(l)}$ if $n=p^l$ and $0$ otherwise, with $\tilde{K}_{n-1}^{(l)}$ as defined in the present proposition. The sign of the function $\tilde{K}_{n-1}^{(l)}-1/l$ determines the position of the generalized Chebyshev primes.
\end{comment}

\section{The Chebyshev and Riemann primes}

The next subsection relates the definition of the Chebyshev primes to the explicit Von Mangoldt formula. The following one puts in perspective the link of the Chebyshev primes to RH through 
the introduction of the so-called {\it Riemann primes}.

\subsection{The Chebyshev primes and the Von Mangoldt explicit formula}
\label{sec14}

\begin{figure}
\centering
\includegraphics[]{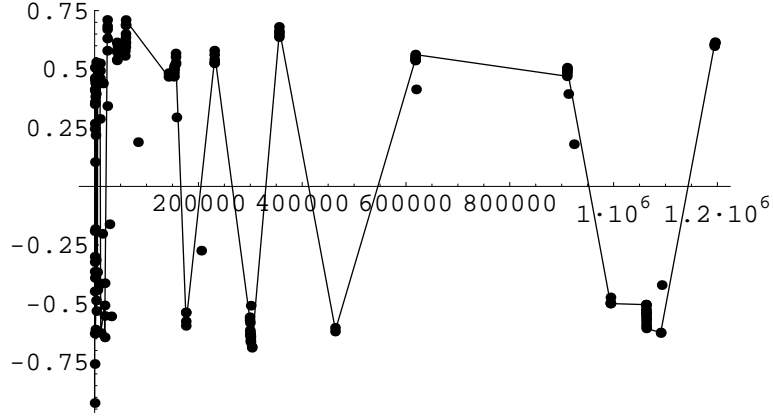}
\caption{The function $(\psi(x)-x)/\sqrt{x}$ at the Riemann primes of the $\psi$-type and index $1$ to $4$. Points of index 1 are joined.}
\label{champs}
\end{figure}
\begin{figure}
\centering
\includegraphics[]{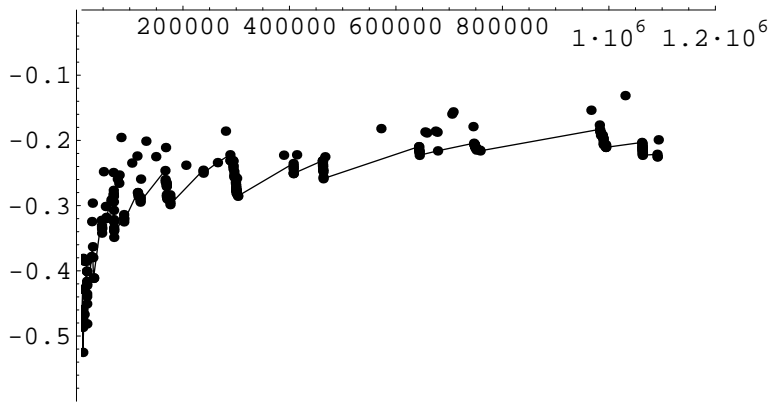}
\caption{The function $(\theta(x)-x)/(\frac{1}{8 \pi}\sqrt{x}\log^2 x)$ at the Riemann primes of the $\theta$-type and index $1$ to $4$. Points of index 1 are joined.}
\label{champstheta}
\end{figure}

From corollary \ref{chebn}, one observes that the oscillations of the function $\psi(x)-x$ around $0$ are intimely related to the existence of Chebyshev primes. 

\begin{proposition}\label{pnpnminus1}
If $p_n$ is a Chebyshev prime (of index $1$), then $\psi(p_n)>p_n$. In the other direction, if $\psi(p_n-1)>p_n$, then $p_n$ is a Chebyshev prime (of index $1$).
\end{proposition}

\begin{proof}
The proposition \ref{pnpnminus1} follows from the inequalities (analogous to that of lemma \ref{encad})

$$\frac{\log p_{n+1}}{\log \psi(p_{n+1})} \le K_n\le \frac{\log p_{n+1}}{\log \psi(p_{n+1}-1)}.$$

\end{proof}

In what regards the position of the (generalized) Chebyshev primes, our numerical experiments lead to

\begin{conjecture}\label{defMan}
Let $\psi_0(x)$ be the normalized Chebyshev function. A Chebyshev prime of index $l$ is defined as a prime $p_m$ satisfying the inequality $\psi_0(p_m^l)>p_m^l$.
\end{conjecture}

\begin{comment}
It is straigthforward to recover the known sequence of Chebyshev primes (already obtained from the definition \ref{def1} or the conjecture \ref{newdef}), from the new conjecture \ref{defMan}. Thus, Chebyshev primes (of index $1$) $\mbox{Ch}_m$ are those primes $p_m$ satisfying $\psi_0(p_m)>p_m$. Similarly, generalized Chebyshev primes of index $l>1$ (obtained from the definition \ref{general}, or the conjecture \ref{newdef2}) also follow from the conjecture \ref{defMan}.


\end{comment}

\subsection{Riemann hypothesis and the Riemann primes}
\label{Riemannp}
Under RH, one has the inequality \cite{Edwards74}
$$|\psi(x)-x|=O(x^{\frac{1}{2}+\epsilon_0})~\mbox{for}~\mbox{every}~\epsilon_0>0,$$
and alternative upper bounds exist in various ranges of values of $x$ \cite{Schoen76}. In the following, we specialize on bounds for $\psi(x)-x$ at power of primes $x=p_n^l$.

\begin{definition}
The champions (left to right maxima) of the function $|\psi(p_n^l)-p_n^l|$ are called here {\it Riemann primes of the $\psi$-type} and index $l$.
\end{definition}

\begin{comment}
One numerically gets the Riemann primes of the $\psi$-type and index $1$ \cite[Sequence A197185]{Sloane} 

\noindent
$\{2,59,73,97,109,113,199,283,463,467,661,1103,1109,1123,1129,1321,1327,\ldots\},$

\noindent
the Riemann primes of the $\psi$-type and index $2$ \cite[Sequence A197186]{Sloane}

\noindent
$\{2,17,31,41,53,101,109,127,139,179,397,419,547,787,997,1031\ldots\},$ 

\noindent
the Riemann primes of the $\psi$-type and index $3$ \cite[Sequence A197187]{Sloane}

\noindent
$\{2,3,5,7,11,13,17,29,59,67,97,103,\ldots\}$ 

\noindent
and the Riemann primes of the $\psi$-type and index $4$ \cite[Sequence A197188]{Sloane}

\noindent
$\{2,5,7,11,13,17,31,\ldots\}.$ 

Clearly, the subset of the Riemann primes of the $\psi$-type such that $\psi_0(p_n^l)>p_n^l$ belongs to the set of Chebyshev primes of the corresponding index $l$.
Since the Riemann primes of the $\psi$-type maximize $\psi(x)-x$, it is useful to plot the ratio $r^{(l)}=(\psi(p_n^l)-p_n^l)/\sqrt {p_n^l}$. Fig. \ref{champs} illustrates this dependence for the Riemann primes of index $1$ to $4$.
One finds that the absolute ratio $|r^{(l)}|$ decreases with the index $l$: this corresponds to the points of lowest amplitude in Fig. \ref{champs}.
\end{comment}

Under RH, one has the inequality \cite[Theorem 10]{Schoen76}
$$|\theta(x)-x| < \frac{1}{8 \pi} \sqrt{x} \log^2 x$$
In the following, we specialize on bounds for $\theta(x)-x$ at power of primes $x=p_n^l$.

\begin{definition}
The champions (left to right maxima) of the function $|\theta(p_n^l)-p_n^l|$ are called here {\it Riemann primes of the $\theta$-type} and index $l$.
\end{definition}

\begin{comment}
One numerically gets the Riemann primes of the $\theta$-type and index $1$ \cite[Sequence A197297]{Sloane} 

\noindent
$\{2, 5, 7, 11, 17, 29, 37, 41, 53, 59, 97, 127, 137, 149, 191, 223, 307, 331, 337, 347, 419,\ldots\},$

\noindent
the Riemann primes of the $\theta$-type and index $2$ \cite[Sequence A197298]{Sloane}

\noindent
$\{2, 3, 5, 7, 11, 13, 17, 19, 23, 29, 31, 37, 43, 47, 59, 73, 97, 107, 109, 139, 179, 233, 263, \ldots\},$ 

\noindent
the Riemann primes of the $\theta$-type and index $3$ \cite[Sequence A197299]{Sloane}

\noindent
$\{2, 3, 5, 7, 13, 17, 23, 31, 37, 41, 43, 47, 53, 59, 67, 73, 83, 89, 101, 103, \ldots\}$ 

\noindent
and the Riemann primes of the $\theta$-type and index $4$ \cite[Sequence A197300]{Sloane}

\noindent
$\{2, 3, 5, 7, 11, 13, 17, 19, 23, 29,\ldots\}.$ 

The Riemann primes of the $\theta$-type maximize $\theta(x)-x$. In Fig. \ref{champstheta}, we plot the ratio $s^{(l)}=(\theta(p_n^l)-p_n^l)/(\frac{1}{8 \pi} \sqrt{p_n^l} \log^2 p_n^l)$ at the Riemann primes of index $1$ to $4$.
Again one finds that the absolute ratio $|s^{(l)}|$ decreases with the index $l$: this corresponds to the points of lowest amplitude in Fig. \ref{champstheta}.
\end{comment}

 In the future, it will be useful to approach the proof of RH thanks to the Riemann primes.

\section{An efficient prime counting function}
\noindent

\begin{figure}
\centering
\includegraphics[]{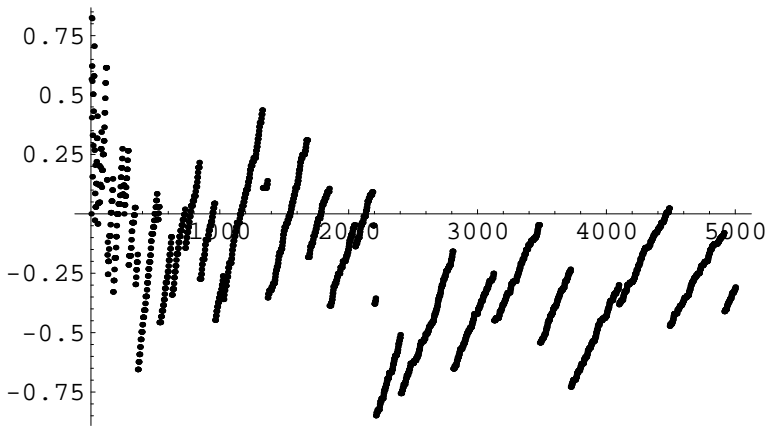}
\caption{A plot of the function $\eta_3(x)$.}
\end{figure}

In this section, one finds that the Riemann prime counting function \cite{Weinstein} $R(x)=\sum_{n=1}^\infty \frac{\mu(n)}{n} \mbox{li}(x^{1/n}) \sim \pi(x)$ may be much improved by replacing it by $R[\psi(x)]$. One denotes $\eta_N(x)=\sum_{n=1}^N \frac{\mu(n)}{n}\mbox{li}[\psi(x)^{1/n}]-\pi(x)$, $N\ge 1$, the offset in the new prime counting function. Indeed, $\eta_1(x)=\epsilon_{\psi(x)}$.

\begin{table}[ht]
\caption{Upper part of the table: maximum error $\eta_{\mbox{max}}$ in the new prime counting function for $x<10^4$ (left hand part) in comparison to the maximum error using the Riemann prime counting function (right hand part).
Lower part of the table: as above in the range $x<10^5$.}\label{table1}
\renewcommand\arraystretch{1.5}
\noindent\[
\begin{array}{|r|r|r|r|r|}
\hline
N  & x_{\mbox{max}} & \eta_{\mbox{max}} & x_{\mbox{max}} & (R-\pi)_{\mbox{max}}\\
\hline
$3$  & $6889$  &  $1.118$     & $7450$   &$6.174$   \\
$4$  & $6889$  &  $1.118$  & $7450$   &$6.174$ \\
$5$  & $1330$  &  $-1.061 $  & $9859$   &$-5.506$ \\
$6$  &  $7$   &  $ -0.862$  & $7450$   &$5.879$  \\
$7$  &  $1330$ &  $-0.936$   & $9949$   &$-5.609$ \\
$10$ &  $7$    & $-0.884$ & $7450$   &$5.661$ \\
$50$ & $1330$  & $-0.885$ & $9949$   &$-5.557$ \\
\hline
$3$  & $80090$ & $1.840$     & $87888$   &$15.304$   \\
$10$ & $49727$ & $-1.158$     & $59797$   &$-15.729$   \\
\hline
\end{array}
\]
\end{table}

\begin{table}[ht]
\caption{Gauss's and Riemann's approximation and the approximation $\eta_3(x)$. Compare table III, p. 35 in \cite{Edwards74}.}\label{table3}
\renewcommand\arraystretch{1.5}
\noindent\[
\begin{array}{|r|r|r|r|}
\hline
x  & \mbox{Planat $\&$ Sol\'e}~\mbox{error} & \mbox{Riemann's}~ \mbox{error}& \mbox{Gauss's}~\mbox{ error}\\
\hline
1,000,000  & 0.79 &  30     & 130     \\
2,000,000  & -0.13 & -8.0     & 121     \\
3,000,000  & 1.83 &  1.8     & 121     \\
4,000,000  & 1.28 &  35     & 130     \\
5,000,000  & 0.36 &  -62     & 121     \\
6,000,000  & 2.91 &  25     & 121     \\
7,000,000  & 0.03 &  -36     & 130     \\
8,000,000  & 2.99 &  -4.7     & 121     \\
9,000,000  & 1.73 &  -51     & 121     \\
10,000,000  & -0.37 &    90  &    339  \\
\hline
\end{array}
\]
\end{table}

By definition, the negative jumps in the function $\eta_N(x)$ may only occur at $x+1 \in \mathcal{P}$. For $N=1$, they occur at primes $p \in \mathcal{\mbox{Ch}}$ (the Chebyshev primes: see definition \ref{def1}). For $N>1$, negative jumps are numerically found to occur at all $x+1 \in \mathcal{P}$ with an amplitude decreasing to zero. We are led to the conjecture

\begin{conjecture}\label{jumpseta1}
Let $\eta_N(x)=\sum_{n=1}^N \frac{\mu(n)}{n}\mbox{li}[\psi(x)^{1/n}]-\pi(x)$, $N >1$. Negative jumps of the function $\eta_N(x)$ occur at all primes $x+1 \in \mathcal{P}$ and $\mbox{lim}_{p \rightarrow \infty}[\eta_N(p)-\eta_N(p-1)]=0$.
\end{conjecture}

More generally, the jumps of $\eta_N(x)$ at power of primes are described by the following 

\begin{conjecture}\label{jumpseta2}
Let $\eta_N(x)$ be as in conjecture \ref{jumpseta1}. Positive jumps of the function $\eta_N(x)$ occur at all power of primes $x+1=p^l$, $p \in \mathcal{P}$ and $l>1$.  Moreover, the jumps are such that 
$\eta_N(p^l)-\eta_N(p^l-1)-1/l>0$ and  $\mbox{lim}_{p \rightarrow \infty}\eta_N(p^l)-\eta_N(p^l-1)-1/l=0$
\end{conjecture}

A sketch of the function $\eta_3(n)$ (for $2<n<1500$) is given in Fig. 2. One easily detects the large  positive jumps at $n=p^2$ ($p \in \mathcal{P}$), the intermediate positive jumps at $n=p^l$ ($l>2$), and the (very small) negative jumps at primes $p$. This plot can be compared to that of the function $R(n)-\pi(n)$ displayed in \cite{Weinstein}.

\begin{comment}
The arithmetical structure of $\eta_N(x)$ just described leads to $|\eta_N(x)|<\eta_{\mbox{max}}$ when $N \ge 3$. Table 1 represents the maximum value $\eta_{\mbox{max}}$ that is reached and the position $x_{\mbox{max}}$ of the extremum, for several small values of $N$ and $x<10^5$. Thus, the function $\sum_{n=1}^N \frac{\mu(n)}{n}\mbox{li}[\psi(x)^{1/n}]$ is a good prime counting function with only a few terms in the summation. 
This is about a fivefold improvement of the accuracy obtained with the standard Riemann prime counting function $R(x)$ (in the range $x<10^4$) and an even better improvement when $x>10^4$, already with three terms in the expansion. 
Another illustration of the efficiency of the calculation based on $\mbox{li}[\psi(x)]$ is given in Table \ref{table3}, that displays values of $\eta_3(x)$ at multiples of $10^6$.

It is known that $R(x)$ converges for any $x$ and may also be written as the Gram series \cite{Weinstein} $R(x)=1+\sum_{k=1^\infty} \frac{(\log x)^k}{k! k \zeta(k+1)}$. A similar formula is not established here.
\end{comment}

\section*{Conclusion} This work sheds light on the structure and the distribution of the generalized Chebyshev primes $\mbox{Ch}_n^{l}$ arising from the jumps of the function $\mbox{li}[\psi(x)]$. It is inspired by Robin's work \cite{Robin84} relating the sign of the functions $\epsilon_{\theta(x)}$ and $\epsilon_{\psi(x)}$ to RH \cite{Robin84}.  Our most puzzling observation is that the non-trivial zeros $\rho$ of the Riemann zeta function are mirrored  in the (generalized) Chebyshev primes, whose existence at infinity crucially depends on the Littlewood's oscillation theorem. In addition, a new accurate prime counting function, based on $\mbox{li}[\psi(x)]$ has been proposed. Future work should concentrate on an effective analytic map between the the zeros $\rho$ and the sequence $\mbox{Ch}_n$, in the spirit of our conjecture \ref{defMan}, and of our approach of RH through the Riemann primes.

\bibliographystyle{amsplain}


\subsection*{The Chebyshev primes (of index $1$) not exceeding $10^5$}
\tiny

[ 109, 113, 139, 181, 197, 199, 241, 271, 281, 283, 293, 313, 317, 443, 449, 461, 463, 467, 479,
491, 503, 509, 523, 619, 643, 647, 653, 659, 661, 673, 677, 683, 691, 701, 761, 769, 773, 829, 859,
863, 883, 887, 1033, 1039, 1049, 1051, 1061, 1063, 1069, 1091, 1093, 1097, 1103, 1109, 1117, 1123,
1129, 1153, 1231, 1237, 1301, 1303, 1307, 1319, 1321, 1327, 1489, 1493, 1499, 1511, 1571, 1579,
1583, 1601, 1607, 1609, 1613, 1619, 1621, 1627, 1637, 1657, 1663, 1667, 1669, 1693, 1697, 1699,
1709, 1721, 1723, 1733, 1741, 1747, 1753, 1759, 1783, 1787, 1789, 1801, 1811, 1877, 1879, 1889,
1907, 1913, 2089, 2113, 2141, 2143, 2153, 2161, 2297, 2311, 2351, 2357, 2381, 2383, 2389, 2393,
2399, 2411, 2417, 2423, 2437, 2441, 2447, 2459, 2467, 2473, 2477, 2557, 2711, 2713, 2719, 2729,
2731, 2741, 2749, 2753, 2767, 2777, 2789, 2791, 2797, 2801, 2803, 2819, 2833, 2837, 2843, 2851,
2857, 2861, 2879, 2887, 2897, 2903, 2909, 2917, 2927, 2939, 2953, 2957, 2963, 2969, 2971, 3001,
3011, 3019, 3023, 3037, 3041, 3049, 3061, 3067, 3083, 3089, 3559, 3583, 3593, 3617, 3623, 3631,
3637, 3643, 3659, 3673, 3677, 3691, 3697, 3701, 3709, 3719, 3727, 3733, 3739, 3761, 3767, 3769,
3779, 3793, 3797, 3803, 3823, 3833, 3851, 3853, 3863, 3877, 3881, 3889, 3911, 3917, 3919, 3923,
3929, 3931, 3943, 3947, 3967, 4007, 4013, 4019, 4021, 4027, 4049, 4051, 4057, 4073, 4079, 4091,
4093, 4099, 4111, 4127, 4129, 4133, 4139, 4153, 4157, 4159, 4177, 4219, 4231, 4241, 4243, 4253,
4259, 4261, 4271, 4273, 4283, 4289, 4297, 4327, 4339, 4349, 4357, 4363, 4373, 4519, 4523, 4549,
4567, 4651, 4657, 4663, 4673, 4679, 4691, 4733, 4801, 4817, 5009, 5011, 5021, 5023, 5039, 5051,
5059, 5081, 5087, 5099, 5101, 5107, 5113, 5119, 5237, 5507, 5521, 5527, 5531, 5573, 5581, 5591,
5659, 5669, 5693, 5701, 5711, 5717, 5743, 5749, 5827, 5843, 5849, 5851, 5857, 5861, 5867, 5869,
5879, 5881, 5897, 5903, 5923, 5927, 5939, 5953, 6271, 6277, 6287, 6299, 6301, 6311, 6317, 6323,
6329, 6337, 6343, 6353, 6359, 6361, 6367, 6373, 6379, 6389, 6397, 6421, 6427, 6449, 6451, 6469,
6473, 6481, 6491, 6571, 6577, 6581, 6599, 6607, 6619, 6703, 6709, 6719, 6733, 6737, 6793, 6803,
6829, 6833, 6841, 6857, 6863, 6869, 6871, 6883, 6899, 6907, 6911, 6917, 6949, 6959, 6961, 6967,
6971, 6977, 6983, 6991, 6997, 7001, 7013, 7019, 7027, 7039, 7043, 7057, 7069, 7079, 7103, 7109,
7121, 7127, 7129, 7151, 7159, 7213, 7219, 7229, 7237, 7243, 7247, 7253, 7333, 7351, 7583, 7589,
7591, 7603, 7607, 7621, 7643, 7649, 7673, 7681, 7687, 7691, 7699, 7703, 7717, 7723, 7727, 7741,
7753, 7757, 7759, 7789, 7793, 7879, 7883, 7927, 7933, 7937, 7949, 7951, 7963, 8297, 8839, 8849,
8863, 8867, 8887, 8893, 9013, 9049, 9059, 9067, 9241, 9343, 9349, 9419, 9421, 9431, 9433, 9437,
9439, 9461, 9463, 9467, 9473, 9479, 9491, 9497, 9511, 9521, 9533, 9539, 9547, 9551, 9587, 9601,
9613, 9619, 9623, 9629, 9631, 9643, 9649, 9661, 9677, 9679, 9689, 9697, 9719, 9721, 9733, 9739,
9743, 9749, 9767, 9769, 9781, 9787, 9791, 9803, 9811, 9817, 9829, 9833, 9839, 9851, 9857, 9859,
9871, 9883, 9887, 9901, 9907, 9923, 9929, 9931, 9941, 9949, 9967, 9973, 10007, 10009, 10039, 10069,
10079, 10091, 10093, 10099, 10103, 10111, 10133, 10139, 10141, 10151, 10159, 10163, 10169, 10177,
10181, 10193, 10211, 10223, 10243, 10247, 10253, 10259, 10267, 10271, 10273, 10289, 10301, 10303,
10313, 10321, 10331, 10333, 10337, 10343, 10357, 10369, 10391, 10399, 10427, 10429, 10433, 10453,
10457, 10459, 10463, 10477, 10487, 10499, 10501, 10513, 10529, 10531, 10559, 10567, 10601, 10607,
10613, 10627, 10631, 10639, 10651, 10657, 10663, 10667, 10687, 10691, 10709, 10711, 10723, 10729,
10733, 10739, 10753, 10771, 10781, 10789, 10799, 10859, 10861, 10867, 10883, 10889, 10891, 10903,
10909, 10939, 10949, 10957, 10979, 10987, 10993, 11003, 11119, 11173, 11177, 12547, 12553, 12577,
12583, 12589, 12601, 12611, 12613, 12619, 12637, 12641, 12647, 12653, 12659, 12671, 12689, 12697,
12703, 12713, 12721, 12739, 12743, 12757, 12763, 12781, 12791, 12799, 12809, 12821, 12823, 12829,
12841, 12853, 12911, 12917, 12919, 12923, 12941, 12953, 12959, 12967, 12973, 12979, 12983, 13001,
13003, 13007, 13009, 13033, 13037, 13043, 13049, 13063, 13093, 13099, 13103, 13109, 13121, 13127,
13147, 13151, 13159, 13163, 13171, 13177, 13183, 13187, 13217, 13219, 13229, 13241, 13249, 13259,
13267, 13291, 13297, 13309, 13313, 13327, 13331, 13337, 13339, 13367, 13381, 13399, 13411, 13417,
13421, 13441, 13451, 13457, 13463, 13469, 13477, 13487, 13499, 13513, 13523, 13537, 13697, 13709,
13711, 13721, 13723, 13729, 13751, 13757, 13759, 13763, 13781, 13789, 13799, 13807, 13829, 13831,
13841, 13859, 13877, 13879, 13883, 13901, 13903, 13907, 13913, 13921, 13931, 13933, 13963, 13967,
13999, 14009, 14011, 14029, 14033, 14051, 14057, 14071, 14081, 14083, 14087, 14107, 14783, 14831,
14851, 14869, 14879, 14887, 14891, 14897, 14947, 14951, 14957, 14969, 14983, 15289, 15299, 15307,
15313, 15319, 15329, 15331, 15349, 15359, 15361, 15373, 15377, 15383, 15391, 15401, 15413, 15427,
15439, 15443, 15451, 15461, 15467, 15473, 15493, 15497, 15511, 15527, 15541, 15551, 15559, 15569,
15581, 15583, 15601, 15607, 15619, 15629, 15641, 15643, 15647, 15649, 15661, 15667, 15671, 15679,
15683, 15727, 15731, 15733, 15737, 15739, 15749, 15761, 15767, 15773, 15787, 15791, 15797, 15803,
15809, 15817, 15823, 15859, 15877, 15881, 15887, 15889, 15901, 15907, 15913, 15919, 15923, 15937,
15959, 15971, 15973, 15991, 16001, 16007, 16033, 16057, 16061, 16063, 16067, 16069, 16073, 16087,
16091, 16097, 16103, 16111, 16127, 16139, 16141, 16183, 16187, 16189, 16193, 16217, 16223, 16229,
16231, 16249, 16253, 16267, 16273, 16301, 16319, 16333, 16339, 16349, 16361, 16363, 16369, 16381,
16411, 16417, 16421, 16427, 16433, 16447, 16451, 16453, 16477, 16481, 16487, 16493, 16519, 16529,
16547, 16553, 16561, 16567, 16573, 16607, 16633, 16651, 16657, 16661, 16673, 16691, 16693, 16699,
16703, 16729, 16741, 16747, 16759, 16763, 17033, 17041, 17047, 17053, 17123, 17207, 17209, 17393,
17401, 17419, 17431, 17449, 17471, 17477, 17483, 17489, 17491, 17497, 17509, 17519, 17539, 17551,
17569, 17573, 17579, 17581, 17597, 17599, 17609, 17623, 17627, 17657, 17659, 17669, 17681, 17683,
17713, 17749, 17761, 17791, 17929, 17939, 17959, 17977, 17981, 17987, 17989, 18013, 18047, 18049,
18059, 18061, 18077, 18089, 18097, 18121, 18127, 18131, 18133, 18143, 18149, 18169, 18181, 18191,
18199, 18211, 18217, 18223, 18229, 18233, 18251, 18253, 18257, 18269, 18287, 18289, 18301, 18307,
18311, 18313, 18329, 18341, 18353, 18367, 18371, 18379, 18397, 18401, 18413, 18427, 18433, 18439,
18443, 18451, 18457, 18461, 18481, 18493, 18503, 18517, 18521, 18523, 18539, 18541, 18553, 18583,
18587, 18593, 18617, 18637, 18661, 18671, 18679, 18691, 18701, 18713, 18719, 18731, 18743, 18749,
18757, 18773, 18787, 18793, 18797, 18803, 20149, 20177, 20183, 20407, 20411, 20443, 21601, 21611,
21613, 21617, 21649, 21661, 21821, 21841, 21851, 21859, 21863, 21871, 21881, 21893, 21911, 22039,
22051, 22067, 22073, 22079, 22091, 22093, 22109, 22111, 22123, 22129, 22133, 22147, 22153, 22157,
22159, 22171, 22189, 22193, 22229, 22247, 22259, 22271, 22273, 22277, 22279, 22283, 22291, 22303,
22307, 22343, 22349, 22367, 22369, 22381, 22391, 22397, 22409, 22433, 22441, 22447, 22453, 22469,
22481, 22483, 22501, 22511, 22543, 22549, 22567, 22571, 22573, 22643, 22651, 22697, 22699, 22709,
22717, 22721, 22727, 22739, 22741, 22751, 22769, 22777, 22783, 22787, 22807, 22811, 22817, 22853,
22859, 22861, 22871, 22877, 22901, 22907, 22921, 22937, 22943, 22961, 22963, 22973, 22993, 23003,
23011, 23017, 23021, 23027, 23029, 23039, 23041, 23053, 23057, 23059, 23063, 23071, 23081, 23087,
23099, 23117, 23131, 23143, 23159, 23167, 23173, 23189, 23197, 23201, 23203, 23209, 23227, 23251,
23269, 23279, 23291, 23293, 23297, 23311, 23321, 23327, 23333, 23339, 23357, 23369, 23371, 23399,
23417, 23431, 23447, 23459, 23473, 23563, 23567, 23581, 23593, 23599, 23603, 23609, 23623, 23627,
23629, 23633, 23663, 23669, 23671, 23677, 23687, 23689, 23719, 23741, 23743, 23747, 23753, 23761,
23767, 23773, 23789, 23801, 23813, 23819, 23827, 23831, 23833, 23857, 23869, 23873, 23879, 23887,
23893, 23899, 23909, 23911, 23917, 23929, 23957, 23971, 23977, 23981, 23993, 24001, 24007, 24019,
24023, 24029, 24043, 24049, 24061, 24071, 24077, 24083, 24091, 24097, 24103, 24107, 24109, 24113,
24121, 24133, 24137, 24151, 24169, 24179, 24181, 24197, 24203, 24223, 24229, 24239, 24247, 24251,
24281, 24317, 24329, 24337, 24359, 24371, 24373, 24379, 24391, 24407, 24413, 24419, 24421, 24439,
24443, 24469, 24473, 24481, 24499, 24509, 24517, 24527, 24533, 24547, 24551, 24571, 24593, 24611,
24623, 24631, 24659, 24671, 24677, 24683, 24691, 24697, 24709, 24733, 24749, 24763, 24767, 24781,
24793, 24799, 24809, 24821, 24851, 24859, 24877, 24923, 24977, 24979, 24989, 25037, 25171, 25189,
25261, 25373, 25463, 25469, 25471, 25951, 26003, 26021, 26029, 26041, 26053, 26263, 26267, 26297,
26317, 26321, 26339, 26347, 26357, 26399, 26407, 26417, 26423, 26431, 26437, 26449, 26459, 26479,
26489, 26497, 26501, 26513, 26713, 26717, 26723, 26729, 26731, 26737, 26759, 26777, 26783, 26801,
26813, 26821, 26833, 26839, 26849, 26861, 26863, 26879, 26881, 26891, 26893, 26903, 26921, 26927,
26947, 26951, 26953, 26959, 26981, 26987, 26993, 27011, 27017, 27031, 27043, 27059, 27061, 27067,
27073, 27077, 27091, 27103, 27107, 27109, 27127, 27143, 27179, 27191, 27197, 27211, 27241, 27259,
27271, 27277, 27281, 27283, 27299, 27823, 27827, 27851, 27961, 27967, 28001, 28031, 28099, 28109,
28111, 28123, 28627, 28631, 28643, 28649, 28657, 28661, 28663, 28669, 28687, 28697, 28703, 28711,
28723, 28729, 28751, 28753, 28759, 28771, 28789, 28793, 28807, 28813, 28817, 28837, 28843, 28859,
28867, 28871, 28879, 28901, 28909, 28921, 28927, 28933, 28949, 28961, 28979, 29009, 29017, 29021,
29023, 29027, 29033, 29059, 29063, 29077, 29101, 29131, 29137, 29147, 29153, 29167, 29173, 29179,
29191, 29201, 29207, 29209, 29221, 29231, 29243, 29251, 29269, 29287, 29297, 29303, 29311, 29327,
29333, 29339, 29347, 29363, 29383, 29387, 29389, 29399, 29401, 29411, 29423, 29429, 29437, 29443,
29453, 29473, 29483, 29501, 29527, 29531, 29537, 29569, 29573, 29581, 29587, 29599, 29611, 29629,
29633, 29641, 29663, 29669, 29671, 29683, 29761, 31277, 31337, 31397, 32611, 32621, 32653, 32803,
32843, 32999, 33029, 33037, 33049, 33053, 33071, 33073, 33083, 33091, 33107, 33113, 33119, 33151,
33161, 33179, 33181, 33191, 33199, 33203, 33211, 33223, 33247, 33349, 33353, 33359, 33377, 33391,
33403, 33409, 33413, 33427, 33469, 33479, 33487, 33493, 33503, 33521, 33529, 33533, 33547, 33563,
33569, 33577, 33581, 33587, 33589, 33599, 33601, 33613, 33617, 33619, 33623, 33629, 33637, 33641,
33647, 33679, 33703, 33713, 33721, 33739, 33749, 33751, 33757, 33767, 33769, 33773, 33791, 33797,
33809, 33811, 33827, 33829, 33851, 33857, 33863, 33871, 33889, 33893, 33911, 33923, 33931, 33937,
33941, 33961, 33967, 33997, 34019, 34031, 34033, 34039, 34057, 34061, 34127, 34129, 34141, 34147,
34157, 34159, 34171, 34183, 34211, 34213, 34217, 34231, 34259, 34261, 34267, 34273, 34283, 34297,
34301, 34303, 34313, 34319, 34327, 34337, 34351, 34361, 34367, 34369, 34381, 34403, 34421, 34429,
34439, 34457, 34469, 34471, 34483, 34487, 34499, 34501, 34511, 34513, 34519, 34537, 34543, 34549,
34583, 34589, 34591, 34603, 34607, 34613, 34631, 34649, 34651, 34667, 34673, 34679, 34687, 34693,
34703, 34721, 34729, 34739, 34747, 34757, 34759, 34763, 34781, 34807, 34819, 34841, 34843, 34847,
34849, 34871, 34877, 34883, 34897, 34913, 34919, 34939, 34949, 34961, 34963, 34981, 35023, 35027,
35051, 35053, 35059, 35069, 35081, 35083, 35089, 35099, 35107, 35111, 35117, 35129, 35141, 35149,
35153, 35159, 35171, 35201, 35221, 35227, 35251, 35257, 35267, 35279, 35281, 35291, 35311, 35317,
35323, 35327, 35339, 35353, 35363, 35381, 35393, 35401, 35407, 35419, 35423, 35437, 35447, 35449,
35461, 35491, 35507, 35509, 35521, 35527, 35531, 35533, 35537, 35543, 35569, 35573, 35591, 35593,
35597, 35603, 35617, 35671, 35677, 36011, 36013, 36017, 36037, 36073, 36083, 36109, 36137, 36697,
36709, 36713, 36721, 36739, 36749, 36761, 36767, 36779, 36781, 36787, 36791, 36793, 36809, 36821,
36833, 36847, 36857, 36871, 36877, 36887, 36899, 36901, 36913, 36919, 36923, 36929, 36931, 36943,
36947, 36973, 36979, 36997, 37003, 37013, 37019, 37021, 37039, 37049, 37057, 37061, 37087, 37097,
37117, 37123, 37139, 37159, 37171, 37181, 37189, 37199, 37201, 37217, 37223, 37243, 37253, 37273,
37277, 37307, 37309, 37313, 37321, 37337, 37339, 37357, 37361, 37363, 37369, 37379, 37397, 37409,
37423, 37441, 37447, 37463, 37483, 37489, 37493, 37501, 37507, 37511, 37517, 37529, 37537, 37547,
37549, 37561, 37567, 37571, 37573, 37579, 37589, 37591, 37607, 37619, 37633, 37643, 37649, 37657,
37663, 37691, 37693, 37699, 37717, 37747, 37781, 37783, 37799, 37811, 37813, 37831, 37847, 37853,
37861, 37871, 37879, 37889, 37897, 37907, 37951, 37957, 37963, 37967, 37987, 37991, 37993, 37997,
38011, 38039, 38047, 38053, 38069, 38083, 38239, 38303, 38321, 38327, 38329, 38333, 38351, 38371,
38377, 38393, 38461, 38933, 39163, 39217, 39227, 39229, 39233, 39239, 39241, 39251, 39301, 39313,
39317, 39323, 39341, 39343, 39359, 39367, 39371, 39373, 39383, 39397, 39409, 39419, 39439, 39443,
39451, 39461, 39503, 39509, 39511, 39521, 39541, 39551, 39563, 39569, 39581, 39623, 39631, 39679,
39841, 39847, 39857, 39863, 39869, 39877, 39883, 39887, 39901, 39929, 39937, 39953, 39971, 39979,
39983, 39989, 40009, 40013, 40031, 40037, 40039, 40063, 40099, 40111, 40123, 40127, 40129, 40151,
 40153, 40163, 40169, 40177, 40189, 40193, 40213, 40231, 40237, 40241,
40253, 40277, 40283, 40289, 41269, 41281, 41651, 41659, 41669, 41681, 41687, 41957, 41959, 41969,
41981, 41983, 41999, 42013, 42017, 42019, 42023, 42043, 42061, 42071, 42073, 42083, 42089, 42101,
42131, 42139, 42157, 42169, 42179, 42181, 42187, 42193, 42197, 42209, 42221, 42223, 42227, 42239,
42257, 42281, 42283, 42293, 42299, 42307, 42323, 42331, 42337, 42349, 42359, 42373, 42379, 42391,
42397, 42403, 42407, 42409, 42433, 42437, 42443, 42451, 42457, 42461, 42463, 42467, 42473, 42487,
42491, 42499, 42509, 42533, 42557, 42569, 42571, 42577, 42589, 42611, 42641, 42643, 42649, 42667,
42677, 42683, 42689, 42697, 42701, 42703, 42709, 42719, 42727, 42737, 42743, 42751, 42767, 42773,
42787, 42793, 42797, 42821, 42829, 42839, 42841, 42853, 42859, 42863, 42899, 42901, 42923, 42929,
42937, 42943, 42953, 42961, 42967, 42979, 42989, 43003, 43013, 43019, 43037, 43049, 43051, 43063,
43067, 43093, 43103, 43117, 43133, 43151, 43159, 43177, 43189, 43201, 43207, 43223, 43237, 43261,
43271, 43283, 43291, 43313, 43319, 43321, 43331, 43391, 43397, 43399, 43403, 43411, 43427, 43441,
43451, 43457, 43481, 43487, 43499, 43517, 43541, 43543, 43577, 43579, 43591, 43597, 43607, 43609,
43613, 43627, 43633, 43649, 43651, 43661, 43669, 43691, 43711, 43717, 43721, 43753, 43759, 43777,
43781, 43783, 43787, 43789, 43793, 43801, 43853, 43867, 43891, 43973, 43987, 43991, 43997, 44017,
44021, 44027, 44029, 44041, 44053, 44059, 44071, 44087, 44089, 44101, 44111, 44119, 44123, 44129,
44131, 44159, 44171, 44179, 44189, 44201, 44203, 44207, 44221, 44249, 44257, 44263, 44267, 44269,
44273, 44279, 44281, 44293, 44351, 44357, 44371, 44381, 44383, 44389, 44417, 44453, 44501, 44507,
44519, 44531, 44533, 44537, 44543, 44549, 44563, 44579, 44587, 44617, 44621, 44623, 44633, 44641,
44647, 44651, 44657, 44683, 44687, 44699, 44701, 44711, 44729, 44741, 44753, 44771, 44773, 44777,
44789, 44797, 44809, 44819, 44839, 44843, 44851, 44867, 44879, 44887, 44893, 44909, 44917, 44927,
44939, 44953, 44959, 44963, 44971, 44983, 44987, 45007, 45013, 45053, 45061, 45077, 45083, 45121,
45127, 45131, 45137, 45139, 45161, 45179, 45181, 45191, 45197, 45233, 45263, 45289, 45293, 45307,
45317, 45319, 45329, 45337, 45341, 45343, 45361, 45377, 45389, 45403, 45413, 45427, 45433, 45439,
45503, 45557, 45589, 45599, 45707, 50119, 50123, 50129, 50131, 50147, 50153, 50159, 50177, 50207,
50227, 50231, 50291, 50333, 50341, 50359, 50363, 50377, 50383, 50387, 50411, 50417, 50423, 50441,
50461, 50551, 50593, 50599, 51613, 51647, 51679, 51683, 51691, 51719, 51721, 51829, 51839, 51853,
51859, 51869, 51871, 51893, 51899, 51907, 51913, 51929, 51941, 51949, 51971, 51973, 51977, 51991,
52009, 52021, 52027, 52057, 52067, 52069, 52081, 52103, 52183, 52189, 52201, 52223, 52253, 52259,
52267, 52291, 52301, 52313, 52321, 52391, 53173, 53201, 53239, 53281, 54563, 54581, 54583, 54601,
54617, 54623, 54629, 54631, 54647, 54667, 54673, 54679, 54713, 54721, 54727, 54779, 54787, 54799,
55933, 55949, 56519, 56527, 56531, 56533, 56543, 56569, 56599, 56611, 56633, 56671, 56681, 56687,
56701, 56711, 56713, 56731, 56737, 56747, 56767, 56773, 56779, 56783, 56807, 56809, 56813, 56821,
56827, 56843, 56857, 56873, 56891, 56893, 56897, 56909, 56911, 56921, 56923, 56929, 56941, 56951,
56957, 56963, 56983, 56989, 56993, 56999, 57037, 57041, 57047, 57059, 57073, 57077, 57089, 57097,
57107, 57119, 57131, 57139, 57143, 57149, 57163, 57173, 57179, 57191, 57193, 57203, 57221, 57223,
57241, 57251, 57259, 57269, 57271, 57283, 57287, 57301, 57329, 57331, 57347, 57349, 57367, 57373,
57383, 57389, 57397, 57413, 57427, 57457, 57467, 57487, 57493, 57503, 57527, 57529, 57557, 57559,
57571, 57587, 57593, 57601, 57637, 57641, 57649, 57653, 57667, 57679, 57689, 57697, 57709, 57713,
57719, 57727, 57731, 57737, 57751, 57773, 57781, 57787, 57791, 57793, 57803, 57809, 57829, 57839,
57847, 57853, 57859, 57881, 57899, 57901, 57917, 57923, 57943, 57947, 57973, 57977, 57991, 58013,
58027, 58031, 58043, 58049, 58057, 58061, 58067, 58073, 58099, 58109, 58111, 58129, 58147, 58151,
58153, 58169, 58171, 58189, 58193, 58199, 58207, 58211, 58217, 58229, 58231, 58237, 58243, 58271,
58309, 58313, 58321, 58337, 58363, 58367, 58369, 58379, 58391, 58393, 58403, 58411, 58417, 58427,
58439, 58441, 58451, 58453, 58477, 58481, 58511, 58537, 58543, 58549, 58567, 58573, 58579, 58601,
58603, 58613, 58631, 58657, 58661, 58679, 58687, 58693, 58699, 58711, 58727, 58733, 58741, 58757,
58763, 58771, 58787, 58789, 58831, 58889, 58897, 58901, 58907, 58909, 58913, 58921, 58937, 58943,
58963, 58967, 58979, 58991, 58997, 59009, 59011, 59021, 59023, 59029, 59051, 59053, 59063, 59069,
59077, 59083, 59093, 59107, 59113, 59119, 59123, 59141, 59149, 59159, 59167, 59183, 59197, 59207,
59209, 59219, 59221, 59233, 59239, 59243, 59263, 59273, 59281, 59333, 59341, 59351, 59357, 59359,
59369, 59377, 59387, 59393, 59399, 59407, 59417, 59419, 59441, 59443, 59447, 59453, 59467, 59471,
59473, 59497, 59509, 59513, 59539, 59557, 59561, 59567, 59581, 59611, 59617, 59621, 59627, 59629,
59651, 59659, 59663, 59669, 59671, 59693, 59699, 59707, 59723, 59729, 59743, 59747, 59753, 59771,
59779, 59791, 59797, 59809, 59833, 59863, 59879, 59887, 59921, 59929, 59951, 59957, 59971, 59981,
59999, 60013, 60017, 60029, 60037, 60041, 60077, 60083, 60089, 60091, 60101, 60103, 60107, 60127,
60133, 60139, 60149, 60161, 60167, 60169, 60209, 60217, 60223, 60251, 60257, 60259, 60271, 60289,
60293, 60317, 60331, 60337, 60343, 60353, 60373, 60383, 60397, 60413, 60427, 60443, 60449, 60457,
60493, 60497, 60509, 60521, 60527, 60539, 60589, 60601, 60607, 60611, 60617, 60623, 60631, 60637,
60647, 60649, 60659, 60661, 60679, 60689, 60703, 60719, 60727, 60733, 60737, 60757, 60761, 60763,
60773, 60779, 60793, 60811, 60821, 60859, 60869, 60887, 60889, 60899, 60901, 60913, 60917, 60919,
60923, 60937, 60943, 60953, 60961, 61001, 61007, 61027, 61031, 61043, 61051, 61057, 61091, 61099,
61121, 61129, 61141, 61151, 61153, 61169, 61211, 61223, 61231, 61253, 61261, 61283, 61291, 61297,
61333, 61339, 61343, 61357, 61363, 61379, 61381, 61403, 61409, 61417, 61441, 61469, 61471, 61483,
61487, 61493, 61507, 61511, 61519, 61543, 61547, 61553, 61559, 61561, 61583, 61603, 61609, 61613,
61627, 61631, 61637, 61643, 61651, 61657, 61667, 61673, 61681, 61687, 61703, 61717, 61723, 61729,
61751, 61757, 61781, 61813, 61819, 61837, 61843, 61861, 61871, 61879, 61909, 61927, 61933, 61949,
61961, 61967, 61979, 61981, 61987, 61991, 62003, 62011, 62017, 62039, 62047, 62053, 62057, 62071,
62081, 62099, 62119, 62129, 62131, 62137, 62141, 62143, 62171, 62189, 62191, 62201, 62207, 62213,
62219, 62233, 62273, 62297, 62299, 62303, 62311, 62323, 62327, 62347, 62351, 62383, 62401, 62417,
62423, 62477, 62483, 62497, 62501, 62507, 62533, 62539, 62549, 62563, 62597, 62603, 62617, 62627,
62633, 62639, 62653, 62659, 62683, 62687, 62701, 62731, 62753, 62761, 62773, 62989, 63499, 63533,
63541, 63599, 63601, 63607, 63611, 63617, 63629, 63647, 63649, 63659, 63667, 63671, 63689, 63691,
63697, 63703, 63709, 63719, 63727, 63737, 63743, 63761, 63773, 63781, 63793, 63799, 63803, 63809,
63823, 63839, 63841, 63853, 63857, 63863, 63901, 63907, 63913, 63929, 63949, 63977, 63997, 64007,
64013, 64019, 64033, 64037, 64063, 64067, 64081, 64091, 64109, 64123, 64151, 64153, 64157, 64171,
64187, 64189, 64217, 64223, 64231, 64237, 64271, 64279, 64283, 64301, 64303, 64319, 64327, 64333,
65713, 65717, 65719, 65729, 65731, 65761, 65777, 65789, 65837, 65839, 65843, 65851, 65867, 65881,
65899, 65929, 65957, 65963, 65981, 65983, 65993, 66047, 66071, 66083, 66089, 66103, 66107, 66109,
66137, 66173, 66179, 66191, 67531, 67537, 67547, 67559, 67567, 67577, 67579, 67589, 67601, 67607,
67619, 67631, 67651, 67759, 67763, 67777, 67783, 67789, 67801, 67807, 67819, 67829, 67843, 67853,
67867, 67883, 67891, 67901, 67927, 67931, 67933, 67939, 67943, 67957, 67961, 67967, 67979, 67987,
67993, 68023, 68041, 68053, 68059, 68071, 68087, 68099, 68111, 68113, 68141, 68147, 68161, 68171,
68209, 68213, 68219, 68227, 68239, 68261, 68281, 71483, 71999, 72047, 72053, 72077, 72089, 72091,
72101, 72103, 72109, 72139, 72167, 72169, 72173, 72223, 72227, 72229, 72251, 72253, 72269, 72271,
72277, 72287, 72307, 72313, 72337, 72341, 72353, 72367, 72379, 72383, 72421, 72431, 72497, 72503,
72679, 72689, 72701, 72707, 72719, 72727, 72733, 72739, 72763, 72767, 72911, 72923, 72931, 72937,
72949, 72953, 72959, 72973, 72977, 72997, 73009, 73013, 73019, 73037, 73039, 73043, 73061, 73063,
73079, 73091, 73121, 73127, 73133, 73141, 73181, 73189, 74611, 74623, 74779, 74869, 74873, 74887,
74891, 74897, 74903, 74923, 74929, 74933, 74941, 74959, 75037, 75041, 75403, 75407, 75431, 75437,
75577, 75583, 75619, 75629, 75641, 75653, 75659, 75679, 75683, 75689, 75703, 75707, 75709, 75721,
75731, 75743, 75767, 75773, 75781, 75787, 75793, 75797, 75821, 75833, 75853, 75869, 75883, 75991,
75997, 76001, 76003, 76031, 76039, 76103, 76159, 76163, 76261, 77563, 77569, 77573, 77587, 77591,
77611, 77617, 77621, 77641, 77647, 77659, 77681, 77687, 77689, 77699, 77711, 77713, 77719, 77723,
77731, 77743, 77747, 77761, 77773, 77783, 77797, 77801, 77813, 77839, 77849, 77863, 77867, 77893,
77899, 77929, 77933, 77951, 77969, 77977, 77983, 77999, 78007, 78017, 78031, 78041, 78049, 78059,
78079, 78101, 78139, 78167, 78173, 78179, 78191, 78193, 78203, 78229, 78233, 78241, 78259, 78277,
78283, 78301, 78307, 78311, 78317, 78341, 78347, 78367, 79907, 80239, 80251, 80263, 80273, 80279,
80287, 80317, 80329, 80341, 80347, 80363, 80369, 80387, 80687, 80701, 80713, 80749, 80779, 80783,
80789, 80803, 80809, 80819, 80831, 80833, 80849, 80863, 80911, 80917, 80923, 80929, 80933, 80953,
80963, 80989, 81001, 81013, 81017, 81019, 81023, 81031, 81041, 81043, 81047, 81049, 81071, 81077,
81083, 81097, 81101, 81119, 81131, 81157, 81163, 81173, 81181, 81197, 81199, 81203, 81223, 81233,
81239, 81281, 81283, 81293, 81299, 81307, 81331, 81343, 81349, 81353, 81359, 81371, 81373, 81401,
81409, 81421, 81439, 81457, 81463, 81509, 81517, 81527, 81533, 81547, 81551, 81553, 81559, 81563,
81569, 81611, 81619, 81629, 81637, 81647, 81649, 81667, 81671, 81677, 81689, 81701, 81703, 81707,
81727, 81737, 81749, 81761, 81769, 81773, 81799, 81817, 81839, 81847, 81853, 81869, 81883, 81899,
81901, 81919, 81929, 81931, 81937, 81943, 81953, 81967, 81971, 81973, 82003, 82007, 82009, 82013,
82021, 82031, 82037, 82039, 82051, 82067, 82073, 82129, 82139, 82141, 82153, 82163, 82171, 82183,
82189, 82193, 82207, 82217, 82219, 82223, 82231, 82237, 82241, 82261, 82267, 82279, 82301, 82307,
82339, 82349, 82351, 82361, 82373, 82387, 82393, 82421, 82457, 82463, 82469, 82471, 82483, 82487,
82493, 82499, 82507, 82529, 82531, 82549, 82559, 82561, 82567, 82571, 82591, 82601, 82609, 82613,
82619, 82633, 82651, 82657, 82699, 82721, 82723, 82727, 82729, 82757, 82759, 82763, 82781, 82787,
82793, 82799, 82811, 82813, 82837, 82847, 82883, 82889, 82891, 82903, 82913, 82939, 82963, 82981,
82997, 83003, 83009, 83023, 83047, 83059, 83063, 83071, 83077, 83089, 83093, 83101, 83117, 83137,
83177, 83203, 83207, 83219, 83221, 83227, 83231, 83233, 83243, 83257, 83267, 83269, 83273, 83299,
83311, 83339, 83341, 83357, 83383, 83389, 83399, 83401, 83407, 83417, 83423, 83431, 83437, 83443,
83449, 83459, 83471, 83477, 83497, 83537, 83557, 83561, 83563, 83579, 83591, 83597, 83609, 83617,
83621, 83639, 83641, 83653, 83663, 83689, 83701, 83717, 83719, 83737, 83761, 83773, 83777, 83791,
83813, 83833, 83843, 83857, 83869, 83873, 83891, 83903, 83911, 83921, 83933, 83939, 83969, 83983,
83987, 84011, 84017, 84061, 84067, 84089, 84143, 84181, 84191, 84199, 84211, 84221, 84223, 84229,
84239, 84247, 84263, 84313, 84317, 84319, 84347, 84349, 84391, 84401, 84407, 84421, 84431, 84437,
84443, 84449, 84457, 84463, 84467, 84481, 84499, 84503, 84509, 84521, 84523, 84533, 84551, 84559,
84589, 84629, 84631, 84649, 84653, 84659, 84673, 84691, 84697, 84701, 84713, 84719, 84731, 84737,
84751, 84761, 84787, 84793, 84809, 84811, 84827, 84857, 84859, 84869, 84871, 84919, 85093, 85103,
85109, 85121, 85133, 85147, 85159, 85201, 85213, 85223, 85229, 85237, 85243, 85247, 85259, 85303,
85313, 85333, 85369, 85381, 85451, 85453, 85469, 85487, 85531, 85621, 85627, 85639, 85643, 85661,
85667, 85669, 85691, 85703, 85711, 85717, 85733, 85751, 85837, 85843, 85847, 85853, 85889, 85909,
85933, 86297, 86353, 86357, 86369, 86371, 86381, 86389, 86399, 86413, 86423, 86441, 86453, 86461,
86467, 86477, 86491, 86501, 86509, 86531, 86533, 86539, 86561, 86573, 86579, 86587, 86599, 86627,
86629, 87641, 87643, 87649, 87671, 87679, 87683, 87691, 87697, 87701, 87719, 87721, 87739, 87743,
87751, 87767, 87793, 87797, 87803, 87811, 87833, 87853, 87869, 87877, 87881, 87887, 87911, 87917,
87931, 87943, 87959, 87961, 87973, 87977, 87991, 88001, 88003, 88007, 88019, 88037, 88069, 88079,
88093, 88117, 88129, 90023, 90031, 90067, 90071, 90073, 90089, 90107, 90121, 90127, 90191, 90197,
90199, 90203, 90217, 90227, 90239, 90247, 90263, 90271, 90281, 90289, 90313, 90373, 90379, 90397,
90401, 90403, 90407, 90437, 90439, 90473, 90481, 90499, 90511, 90523, 90527, 90529, 90533, 90547,
90647, 90659, 90679, 90703, 90709, 91459, 91463, 92419, 92431, 92467, 92489, 92507, 92671, 92681,
92683, 92693, 92699, 92707, 92717, 92723, 92737, 92753, 92761, 92767, 92779, 92789, 92791, 92801,
92809, 92821, 92831, 92849, 92857, 92861, 92863, 92867, 92893, 92899, 92921, 92927, 92941, 92951,
92957, 92959, 92987, 92993, 93001, 93047, 93053, 93059, 93077, 93083, 93089, 93097, 93103, 93113,
93131, 93133, 93139, 93151, 93169, 93179, 93187, 93199, 93229, 93239, 93241, 93251, 93253, 93257,
93263, 93281, 93283, 93287, 93307, 93319, 93323, 93329, 93337, 93371, 93377, 93383, 93407, 93419,
93427, 93463, 93479, 93481, 93487, 93491, 93493, 93497, 93503, 93523, 93529, 93553, 93557, 93559,
93563, 93581, 93601, 93607, 93629, 93637, 93683, 93701, 93703, 93719, 93739, 93761, 93763, 93787,
93809, 93811, 93827, 93851, 93893, 93901, 93911, 93913, 93923, 93937, 93941, 93949, 93967, 93971,
93979, 93983, 93997, 94007, 94009, 94033, 94049, 94057, 94063, 94079, 94099, 94109, 94111, 94117,
94121, 94151, 94153, 94169, 94201, 94207, 94219, 94229, 94253, 94261, 94273, 94291, 94307, 94309,
94321, 94327, 94331, 94343, 94349, 94351, 94379, 94397, 94399, 94421, 94427, 94433, 94439, 94441,
94447, 94463, 94477, 94483, 94513, 94529, 94531, 94541, 94543, 94547, 94559, 94561, 94573, 94583,
94597, 94603, 94613, 94621, 94649, 94651, 94687, 94693, 94709, 94723, 94727, 94747, 94771, 94777,
94781, 94789, 94793, 94811, 94819, 94823, 94837, 94841, 94847, 94849, 94873, 94889, 94903, 94907,
94933, 94949, 94951, 94961, 94993, 94999, 95003, 95009, 95021, 95027, 95063, 95071, 95083, 95087,
95089, 95093, 95101, 95107, 95111, 95131, 95143, 95153, 95177, 95189, 95191, 95203, 95213, 95219,
95231, 95233, 95239, 95257, 95261, 95267, 95273, 95279, 95287, 95311, 95317, 95327, 95339, 95369,
95383, 95393, 95401, 95413, 95419, 95429, 95441, 95443, 95461, 95467, 95471, 95479, 95483, 95507,
95527, 95531, 95539, 95549, 95561, 95569, 95581, 95597, 95603, 95617, 95621, 95629, 95633, 95651,
95701, 95707, 95713, 95717, 95723, 95731, 95737, 95747, 95773, 95783, 95789, 95791, 95801, 95803,
95813, 95819, 95857, 95869, 95873, 95881, 95891, 95911, 95917, 95923, 95929, 95947, 95957, 95959,
95971, 95987, 95989, 96001, 96013, 96017, 96043, 96053, 96059, 96079, 96097, 96137, 96149, 96157,
96167, 96179, 96181, 96199, 96211, 96221, 96223, 96233, 96259, 96263, 96269, 96281, 96289, 96293,
96323, 96329, 96331, 96337, 96353, 96377, 96401, 96419, 96431, 96443, 96451, 96457, 96461, 96469,
96479, 96487, 96493, 96497, 96517, 96527, 96553, 96557, 96581, 96587, 96589, 96601, 96643, 96661,
96667, 96671, 96697, 96703, 96731, 96737, 96739, 96749, 96757, 96763, 96769, 96779, 96787, 96797,
96799, 96821, 96823, 96827, 96847, 96851, 96857, 96893, 96907, 96911, 96931, 96953, 96959, 96973,
96979, 96989, 96997, 97001, 97003, 97007, 97021, 97039, 97073, 97081, 97103, 97117, 97127, 97151,
97157, 97159, 97169, 97171, 97177, 97187, 97213, 97231, 97241, 97259, 97283, 97301, 97303, 97327,
97367, 97369, 97373, 97379, 97381, 97387, 97397, 97423, 97429, 97441, 97453, 97459, 97463, 97499,
97501, 97511, 97523, 97547, 97549, 97553, 97561, 97571, 97577, 97579, 97583, 97607, 97609, 97613,
97649, 97651, 97673, 97687, 97711, 97729, 97771, 97777, 97787, 97789, 97813, 97829, 97841, 97843,
97847, 97849, 97859, 97861, 97871, 97879, 97883, 97919, 97927, 97931, 97943, 97961, 97967, 97973,
97987, 98009, 98011, 98017, 98041, 98047, 98057, 98081, 98101, 98123, 98129, 98143, 98179, 98227,
98327, 98389, 98411, 98419, 98429, 98443, 98453, 98459, 98467, 98473, 98479, 98491, 98507, 98519,
98533, 98543, 98561, 98563, 98573, 98597, 98627, 98639, 98641, 98663, 98669, 98717, 98729, 98731,
98737, 98779, 98809, 98893, 98897, 98899, 98909, 98911, 98927, 98929, 98939, 98947, 98953, 98963,
98981, 98993, 98999, 99013, 99017, 99023, 99041, 99053, 99079, 99083, 99089, 99103, 99109, 99119,
99131, 99133, 99137, 99139, 99149, 99173, 99181, 99191, 99223, 99233, 99241, 99251, 99257, 99259,
99277, 99289, 99317, 99347, 99349, 99367, 99371, 99377, 99391, 99397, 99401, 99409, 99431, 99439,
99469, 99487, 99497, 99523, 99527, 99529, 99551, 99559, 99563, 99571, 99577, 99581, 99607, 99611,
99623, 99643, 99661, 99667, 99679, 99689, 99707, 99709, 99713, 99719, 99721, 99733, 99761, 99767,
99787, 99793, 99809, 99817, 99823, 99829, 99833, 99839, 99859, 99871, 99877, 99881, 99901, 99907,
99923, 99929, 99961, 99971, 99989, 99991 ]

\subsection*{The Riemann primes of the $\psi$-type and index $1$, in the range $p_n=[2\ldots1286451]$}
\small
[ 2, 59, 73, 97, 109, 113, 199, 283, 463, 467, 661, 1103, 1109, 1123, 1129, 1321, 1327, 1423, 2657,
2803, 2861, 3299, 5381, 5881, 6373, 6379, 9859, 9931, 9949, 10337, 10343, 11777, 19181, 19207,
19373, 24107, 24109, 24113, 24121, 24137, 42751, 42793, 42797, 42859, 42863, 58231, 58237, 58243,
59243, 59447, 59453, 59471, 59473, 59747, 59753,142231, 142237, 151909, 152851, 152857, 152959,
152993, 153001, 155851, 155861, 155863, 155893, 175573, 175601, 175621, 230357, 230369, 230387,
230389, 230393, 298559, 298579, 298993, 299281, 299311, 299843, 299857, 299933, 300073, 300089,
300109, 300137, 302551, 302831, 355073, 355093, 355099, 355109, 355111, 463157, 463181, 617479,
617731, 617767, 617777, 617801, 617809, 617819, 909907, 909911, 909917, 910213, 910219, 910229,
993763, 993779, 993821, 1062251, 1062293, 1062311, 1062343, 1062469, 1062497, 1062511, 1062547,
1062599, 1062643, 1062671, 1062779, 1062869, 1090681, 1090697, 1194041, 1194047, 1194059, 1195237,
1195247 ]

\subsection*{The Riemann primes of the $\theta$-type and index $1$, in the range $p_n=[2\ldots1536517]$}
\small
[ 2, 5, 7, 11, 17, 29, 37, 41, 53, 59, 97, 127, 137, 149, 191, 223, 307, 331, 337, 347, 419, 541,
557, 809, 967, 1009, 1213, 1277, 1399, 1409, 1423, 1973, 2203, 2237, 2591, 2609, 2617, 2633, 2647,
2657, 3163, 3299, 4861, 4871, 4889, 4903, 4931, 5381, 7411, 7433, 7451, 8513, 8597, 11579, 11617,
11657, 11677, 11777, 14387, 18973, 19001, 19031, 19051, 19069, 19121, 19139, 19181, 19207, 19373,
27733, 30089, 30631, 31957, 32051, 46439, 47041, 47087, 47111, 47251, 47269, 55579, 55603, 64849,
64997, 69109, 69143, 69191, 69337, 69371, 69623, 69653, 69677, 69691, 69737, 69761, 69809, 69821,
69991, 88589, 88643, 88771, 88789, 114547, 114571, 115547, 115727, 119489, 119503, 119533, 119549,
166147, 166541, 166561, 168433, 168449, 168599, 168673, 168713, 168851, 168977, 168991, 169307,
169627, 175391, 175573, 175601, 175621, 237673, 237851, 237959, 264731, 288137, 288179, 288647,
293599, 293893, 293941, 293957, 293983, 295663, 295693, 295751, 295819, 298153, 298559, 298579,
298993, 299261, 299281, 299311, 299843, 299857, 299933, 300073, 300089, 300109, 300137, 302551,
302831, 406969, 407023, 407047, 407083, 407119, 407137, 407177, 461917, 461957, 461971, 462013,
462041, 462067, 462239, 462263, 462307, 462361, 462401, 463093, 463157, 463181, 642673, 642701,
642737, 642769, 643369, 643403, 643421, 643847, 678157, 745931, 747199, 747259, 747277, 747319,
747361, 747377, 747391, 747811, 747827, 748169, 748183, 748199, 748441, 750599, 750613, 750641,
757241, 757993, 982559, 983063, 983113, 984241, 984299, 987793, 987911, 987971, 989059, 989119,
989171, 989231, 989623, 989743, 993679, 993763, 993779, 993821, 1061561, 1062169, 1062197, 1062251,
1062293, 1062311, 1062343, 1062469, 1062497, 1062511, 1062547, 1062599, 1062643, 1062671, 1062779,
1062869, 1090373, 1090681, 1090697 ]

%
%

%
%

\end{document}